\documentclass[10pt,leqno]{amsart}
\usepackage{amssymb,amsthm}
\usepackage{amsmath}
\usepackage{setspace}
\usepackage{dcpic,pictexwd}
\usepackage[all]{xy}
\usepackage[pdftex]{graphicx}
\usepackage{mathrsfs}
\usepackage[pdftex]{hyperref}
\usepackage{bbm}

\theoremstyle{definition}
 \newtheorem{definition}{Definition}[section]

\theoremstyle{plain}

\theoremstyle{plain}
 \newtheorem{theorem}[definition]{Theorem}
 
\theoremstyle{definition}
 \newtheorem{example}[definition]{Example}

\theoremstyle{plain}
 \newtheorem{lemma}[definition]{Lemma}

\theoremstyle{plain}

\theoremstyle{remark}
 \newtheorem{remark}[definition]{Remark}

\theoremstyle{definition}

\theoremstyle{plain}

\newcommand{\Ext}{\mathrm{Ext}}
\newcommand{\End}{\mathrm{End}}

\newcommand{\Hom}{\mathrm{Hom}}
\newcommand{\SHom}{\underline{\mathrm{Hom}}}

\newcommand{\Ca}{\mathcal{C}}
\newcommand{\Fun}{\mathrm{F}}

\newcommand{\Def}{\mathrm{Def}}
\newcommand{\Sets}{\mathrm{Sets}}

\newcommand{\Ob}{\mathrm{Ob}}

\newcommand{\SEnd}{\underline{\End}}
\newcommand{\A}{\Lambda}

\renewcommand{\k}{\Bbbk}

\hypersetup{
    bookmarks=true,         % show bookmarks bar?
    unicode=false,          % non-Latin characters in AcrobatÕs bookmarks
    pdftoolbar=true,        % show AcrobatÕs toolbar?
    pdfmenubar=true,        % show AcrobatÕs menu?
    pdffitwindow=false,     % window fit to page when opened
    pdfstartview={FitH},    % fits the width of the page to the window
    pdftitle={Universal Deformation Rings for Complexes Over Finite-Dimensional Algebras},    % title
    pdfauthor={Jose A. Velez-Marulanda},     % author
    pdfsubject={Derived Categories, Singularity Categories, Universal deformation rings},   % subject of the document
    pdfcreator={Jose A. Velez-Marulanda},   % creator of the document
    pdfproducer={Jose A. Velez-Marulanda}, % producer of the document
    pdfkeywords={Derived Categories} {Singularity Categories}{Universal Deformation Rings}, % list of keywords
    pdfnewwindow=true,      % links in new window
    colorlinks=true,       % false: boxed links; true: colored links
    linkcolor=red,          % color of internal links
    citecolor=blue,        % color of links to bibliography
    filecolor=magenta,      % color of file links
    urlcolor=blue           % color of external links
}
\title[Universal Deformation Rings of Gorenstein-projective Modules]{Universal Deformation Rings of Finitely Generated Gorenstein-Projective Modules over Finite Dimensional Algebras} 

\thanks{The third author was supported by the Release Time for Research Scholarship of the Office of Academic Affairs and by the Faculty 
Research Seed Grant funded by the Office of Sponsored Programs \& Research Administration at the Valdosta State University. The authors were partly supported by CODI and Estrategia  de Sostenibilidad (Universidad de Antioquia, UdeA) and Colciencias-Ecopetrol (no. 0266-2013)}
\author{Viktor Bekkert}
\address{Departamento de Matem\'atica, Instituto de Ci\^encias Exatas, Universidade Federal de Minas Gerais }
\email{bekkert@mat.ufmg.br}  
\author{Hern\'an Giraldo}
\address{Instituto de Matem\'aticas, Universidad de Antioquia, Medell{\'\i}n, Antioquia, Colombia}
\email{hernan.giraldo@udea.edu.co}  
\author{Jos\'e A. V\'elez-Marulanda}
\address{Department of Mathematics, Valdosta State University, Valdosta, GA, U.S.A.}
\email{javelezmarulanda@valdosta.edu (Corresponding author)}


\begin{document}
\renewcommand{\labelenumi}{\textup{(\roman{enumi})}}
\renewcommand{\labelenumii}{\textup{(\roman{enumi}.\alph{enumii})}}
\numberwithin{equation}{section}

%\halfspacing

\begin{abstract}
Let $\k$ be a field of arbitrary characteristic, let $\A$ be a finite dimensional $\k$-algebra, and let $V$ be a finitely generated $\A$-module. F. M. Bleher and the third author previously proved that $V$ has a well-defined versal deformation ring $R(\A,V)$. If the stable endomorphism ring of $V$ is isomorphic to $\k$, they also proved under the additional assumption that $\A$ is self-injective that $R(\A,V)$ is universal. In this paper, we prove instead that if $\A$ is arbitrary but $V$ is Gorenstein-projective then $R(\A,V)$ is also universal when the stable endomorphism ring of $V$ is isomorphic to $\k$. Moreover, we show that singular equivalences of Morita type (as introduced by X. W. Chen and L. G. Sun) preserve the isomorphism classes of versal deformation rings of finitely generated Gorenstein-projective modules over Gorenstein algebras. We also provide examples. In particular, if $\A$ is a monomial algebra in which there is no overlap (as introduced by X. W. Chen, D. Shen and G. Zhou) we prove that every finitely generated indecomposable Gorenstein-projective $\A$-module has a universal deformation ring that is isomorphic to either $\k$ or to $\k[\![t]\!]/(t^2)$.
\end{abstract}
\maketitle
%\tableofcontents

\section{Introduction}\label{int}

Throughout this article, we assume that $\k$ is a fixed field of arbitrary characteristic. We denote by $\hat{\Ca}$ the category of all complete local commutative Noetherian $\k$-algebras with residue field $\k$. In particular, the morphisms in $\hat{\Ca}$ are continuous $\k$-algebra homomorphisms that induce the identity map on $\k$.  Let $\A$ be a fixed finite dimensional $\k$-algebra. For all objects $R\in \Ob(\hat{\Ca})$, we denote by $R\A$ the tensor product of $\k$-algebras $R\otimes_\k\A$. Note that $R\A$ is an $R$-algebra, and if $R$ is an Artinian ring then $R\A$ is also Artinian (both on the left and the right). Let $V$ be a finitely generated left $\A$-module. We denote by $\End_\A(V)$ (resp. by $\SEnd_\A(V)$) the 
endomorphism ring (resp. the stable endomorphism ring) of $V$. We denote by $\Omega V$ the first syzygy of $V$, i.e. $\Omega V$ is the kernel of a projective cover $P(V)\to V$ of $V$ over $\A$, which is unique up to isomorphism. Let $R$ be an arbitrary object in $\hat{\Ca}$. A {\it lift} of $V$ over $R$ is an $R\A$-module $M$ that is free over $R$ together with a $\A$-module isomorphism $\phi: \k\otimes_RM\to V$.   A {\it deformation} of $V$ over $R$ is defined to be an isomorphism class of lifts of $V$ over $R$.  In \cite{blehervelez}, F. M. Bleher and the third author studied deformations and deformation rings of modules for arbitrary finite dimensional $\k$ algebras. In particular, they proved that when $\A$ is a self-injective algebra (i.e. $\A$ is injective as a left $\A$-module) and $V$ is a $\A$-module with finite dimension over $\k$ such that $\SEnd_\A(V)\cong \k$, then $V$ has a universal deformation ring $R(\A,V)$ that is stable under taking syzygies provided that $\A$ is further a Frobenius algebra (i.e. $\A$ and $\Hom_\k(\A,\k)$ are isomorphic as right $\A$-modules). The results in \cite{blehervelez} were used in \cite{bleher9,blehervelez,velez} to study universal deformation rings for certain self-injective algebras which are not Morita equivalent to a block of a group algebra. More recently, it was proved in \cite[Prop. 3.2.6]{blehervelez2} that the isomorphism class of versal deformation rings of modules is preserved by stable equivalences of Morita type (as introduced by M. Brou\'e in \cite{broue}) between self-injective $\k$-algebras. Moreover, in \cite{bleher15}, F. M. Bleher and D. J. Wackwitz studied universal deformation rings of modules over self-injective Nakayama $\k$-algebras. 

Our aim is to study lifts of finitely generated Gorenstein-projective modules over finite dimensional $\k$-algebras.

Following \cite{enochs0,enochs}, we say that a (not necessarily finitely generated) left $\A$-module $W$ is {\it Gorenstein-projective} provided that there exists an acyclic complex of (not necessarily finitely generated) projective left $\A$-modules 
\begin{equation*}
P^\bullet: \cdots\to P^{-2}\xrightarrow{f^{-2}} P^{-1}\xrightarrow{f^{-1}} P^0\xrightarrow{f^0}P^1\xrightarrow{f^1}P^2\to\cdots
\end{equation*}  
such that $\Hom_\A(P^\bullet, \A)$ is also acyclic and $W=\mathrm{coker}\,f^0$.

Following \cite{auslander4} and \cite{avramov2}, we say that the finitely generated left $\A$-module $V$ is of {\it Gorenstein dimension zero} or {\it totally reflexive} provided that $V$ and $\Hom_\A(\Hom_\A(V,\A),\A)$ are isomorphic as left $\A$-modules, and that $\Ext_\A^i(V,\A)=0=\Ext_\A^i(\Hom_\A(V,\A),\A)$
for all $i>0$. It is well-known that finitely generated Gorenstein-projective left $\A$-modules coincide with those that are totally reflexive (see e.g. \cite[\S 2.4]{avramov2}). Following \cite{buchweitz}, we say that $V$ is a {\it (maximal) Cohen-Macaulay} $\A$-module provided that $\Ext_\A^i(V,\A)=0$ for all $i>0$. Recall that $\A$ is said to be a {\it Gorenstein} $\k$-algebra provided that $\A$ has finite injective dimension as a left and right $\A$-module (see \cite{auslander2}). In particular, algebras of finite global dimension as well as self-injective algebras are Gorenstein. It follows from \cite[Prop. 4.1]{auslander3} that if $\A$ is a Gorenstein $\k$-algebra, then finitely generated Gorenstein-projective and (maximal) Cohen-Macaulay left $\A$-modules coincide. However, it follows from an example given by J.I. Miyachi (see \cite[Example A.3]{hoshino-koga}) that in general not all (maximal) Cohen-Macaulay modules over a finite dimensional algebra are Gorenstein-projective. 

The singularity category $\mathcal{D}_{\mathrm{sg}}(\A\textup{-mod})$ of $\A$ is the Verdier quotient of the bounded derived 
category of finitely generated left $\A$-modules $\mathcal{D}^b(\A\textup{-mod})$ by the full subcategory $\mathcal{K}^b(\A\textup{-proj})$ of perfect complexes (see 
\cite{verdier} and e.g. \cite{krause3} for the construction of this quotient). If $\A$ is self-injective, then it follows from \cite[Thm. 2.1]{rickard1} that $\mathcal{D}_{\mathrm{sg}}(\A\textup{-mod})$ is equivalent as a triangulated category to $\A\textup{-\underline{mod}}$, the 
stable category of finitely generated left $\A$-modules. If $\A$ is Gorenstein, then it follows from \cite[Thm. 4.4.1]{buchweitz} (see also \cite[\S4.6]{happel3} in the case when $\k$ is algebraically closed) that $\mathcal{D}_{\mathrm{sg}}(\A\textup{-mod})$ is equivalent as a triangulated 
category to $\A\textup{-\underline{Gproj}}$ the stable category of finitely generated Gorenstein-projective left $\A$-modules. In particular, if $\A$ has finite global dimension, then its singularity category is trivial.

The following definition was introduced by X. W. Chen and L. G. Sun in \cite{chensun}, which was further studied by G. Zhou and A. Zimmermann in \cite{zhouzimm}, as a way of generalizing the concept of stable equivalence of Morita type.  
\begin{definition}\label{defi:3.2}
Let $\A$ and $\Gamma$ be finite dimensional $\k$-algebras, and let $X$ be a $\Gamma$-$\A$-bimodule and $Y$ a $\A$-$\Gamma$-bimodule. We say that $X$ and $Y$ induce a {\it singular equivalence of Morita type} between $\A$ and $\Gamma$ (and that $\A$ and $
\Gamma$ are {\it singularly equivalent of Morita type}) if the following conditions are satisfied:
\begin{enumerate}
\item $X$ is finitely generated and projective as a left $\Gamma$-module and as a right $\A$-module.
\item $Y$ is finitely generated and projective as a left $\A$-module and as a right $\Gamma$-module. 
\item There is a finitely generated $\Gamma$-$\Gamma$-bimodule $Q$ with finite projective dimension such that $X\otimes_\A Y\cong 
\Gamma \oplus Q$ as $\Gamma$-$\Gamma$-bimodules.
\item There is a finitely generated $\A$-$\A$-bimodule $P$ with finite projective dimension such that $Y\otimes_\Gamma X\cong \A\oplus P
$ as $\A$-$\A$-bimodules.
\end{enumerate}
\end{definition}
It follows from \cite[Prop. 2.3]{zhouzimm} that singular equivalences of Morita type induce equivalences of singularity categories.

The goal of this article is to prove the following result.

\begin{theorem}\label{thm01}
Let $\A$ be a finite dimensional $\k$-algebra and let $V$ be a non-zero finitely generated Gorenstein-projective left $\A$-module. 
\begin{enumerate}
\item For all finitely generated projective left $\A$-modules $P$, the versal deformation rings $R(\A,V)$ and $R(\A,V\oplus P)$ are isomorphic in $\hat{\Ca}$. 
\item If $\SEnd_\A(V)=\k$, then the versal deformation ring $R(\A,V)$ is universal.  In particular, the versal deformation ring $R(\A, \Omega V)$ of $\Omega V$ is also universal.  
\item Assume that $\A$ is Gorenstein and let $\Gamma$ be another finite dimensional Gorenstein $\k$-algebra. Assume that ${_\Gamma}X_\A$, ${_\A} Y_\Gamma$ are bimodules that induce a singular equivalence of Morita type as in Definition \ref{defi:3.2} between $\A$ and $\Gamma$. Then $X\otimes_\A V$ is a finitely generated Gorenstein-projective left $\Gamma$-module, and the versal deformation rings $R(\A, V)$ and $R(\Gamma, X\otimes_\A V)$ are isomorphic in $\hat{\Ca}$. 
\end{enumerate} 
\end{theorem}

Since every finitely generated module over a self-injective algebra is Gorenstein-projective and since self-injective algebras are Gorenstein, it follows that parts (i), (ii) and (iii) of Theorem \ref{thm01} provide a generalization of 
\cite[Lemma 3.2.2]{blehervelez2}, \cite[Thm. 2.6(ii)]{blehervelez} and \cite[Prop. 3.2.6]{blehervelez2}, respectively. 

This article is organized as follows. In \S \ref{sec2}, we review some preliminary definitions and properties concerning deformations and (uni)versal deformation rings of modules over finite dimensional algebras as discussed in \cite{blehervelez}, and we review some basic facts concerning finitely generated Gorenstein-projective modules. In \S \ref{section3} we prove Theorem \ref{thm01} (i) and (ii) by carefully adapting some of the ideas in the proofs of \cite[Thm. 2.6]{blehervelez} and \cite[Lemma 3.2.2]{blehervelez2} to our context. In \S \ref{section4}, we prove Theorem \ref{thm01} (iii) by adapting the ideas in the proof of \cite[Prop. 3.2.6]{blehervelez2} to our context.  Finally in \S \ref{section5}, we prove that every finitely generated indecomposable Gorenstein-projective module over a monomial algebra in which there is no overlap (as introduced in \cite{chen-shen-zhou}) has a universal deformation ring, which is isomorphic either to $\k$ or to $\k[\![t]\!]/(t^2)$ (see Theorem \ref{thmmon}). Moreover, we illustrate part (iii) of Theorem \ref{thm01} by discussing an example of two finite dimensional Gorenstein $\k$-algebras $\A$ and $\Gamma$, where both are non-self-injective of infinite global dimension such that $\A$ and $\Gamma$ are stably equivalent of Morita type, and thus singularly equivalent of Morita type as in Definition \ref{defi:3.2} (see Example \ref{exam1}).

For basic concepts concerning Gorenstein-projective modules, we refer the reader to  \cite{auslander3,holmH} (and their references). For basic concepts from the representation theory of algebras such as projective covers, syzygies of modules, stable categories and homological dimension of modules over finite dimensional algebras, we refer the reader to \cite{auslander,curtis,weibel}.

\section{Preliminaries}\label{sec2}
Throughout this section we keep the notation introduced in \S \ref{int}. In particular, $\A$ is a finite dimensional $\k$-algebra, $V$ is a finitely generated $\A$-module, and $R$ is a ring in the category $\hat{\Ca}$.
\subsection{Lifts, deformations, and (uni)versal deformation rings}\label{sec21}

A {\it lift} $(M,\phi)$ 
of $V$ over $R$ is a finitely generated left $R\A$-module $M$ 
that is free over $R$ 
together with an isomorphism of $\A$-modules $\phi:\k\otimes_RM\to V$. Two lifts $(M,\phi)$ and $(M',\phi')$ over $R$ are {\it isomorphic} 
if there exists an $R\A$-module 
isomorphism $f:M\to M'$ such that $\phi'\circ (\k\otimes_R f)=\phi$.
If $(M,\phi)$ is a lift of $V$ over $R$, we  denote by $[M,\phi]$ its isomorphism class and say that $[M,\phi]$ is a {\it deformation} of $V$ 
over $R$. We denote by $\Def_\A(V,R)$ the 
set of all deformations of $V$ over $R$. The {\it deformation functor} corresponding to $V$ is the 
covariant functor $\hat{\Fun}_V:\hat{\Ca}\to \Sets$ defined as follows: for all objects $R$ in $\Ob(\hat{\Ca})$, define $\hat{\Fun}_V(R)=\Def_
\A(V,R)$, and for all morphisms $\alpha:R\to 
R'$ in $\hat{\Ca}$, 
let $\hat{\Fun}_V(\alpha):\Def_\A(V,R)\to \Def_\A(V,R')$ be defined as $\hat{\Fun}_V(\alpha)([M,\phi])=[R'\otimes_{R,\alpha}M,\phi_\alpha]$, 
where $\phi_\alpha: \k\otimes_{R'}
(R'\otimes_{R,\alpha}M)\to V$ is the composition of $\A$-module isomorphisms 
\[\k\otimes_{R'}(R'\otimes_{R,\alpha}M)\cong \k\otimes_RM\xrightarrow{\phi} V.\]  

\begin{remark}\label{rem1}
Some authors consider a weaker notion of deformations. Namely, let $\A$ and $V$ be as above, let $R$ be an object in $\hat{\Ca}$ and let $(M,\phi)$ be a lift of $V$ over $R$. Then the isomorphism class $[M]$ of $M$ as an $R\A$-module is called a {\it weak deformation} of $V$ over $R$ (see e.g. \cite[\S 5.2]{keller} and \cite[Remark 2.4]{blehervelez}). We can also define the weak deformation functor $\hat{\Fun}_V^w: \hat{\Ca}\to\Sets$ which sends an object $R$ in $\hat{\Ca}$ to the set of
weak deformations of $V$ over $R$  and a morphism  $\alpha:R \to R'$  in $\hat{\Ca}$ to the map  $\hat{\Fun}_V^w :\hat{\Fun}_V^w (R)\to \hat{\Fun}_V^w(R')$, which is defined by $\hat{\Fun}_V^w(\alpha)([M]) = [R'\otimes_{R,\alpha}M]$.
In general, a weak deformation of $V$ over $R$ identifies more lifts than a deformation of $V$ over $R$ that respects the isomorphism $\phi$ of a representative $(M,\phi)$. 
\end{remark}

Suppose there exists an object $R(\A,V)$ in $\Ob(\hat{\Ca})$  and a deformation $[U(\A,V), \phi_{U(\A,V)}]$ of $V$ over $R(\A,V)$ with the 
following property. For each $R$ in $\Ob(\hat{\Ca})$ and for all deformations $[M,\phi]$ of $V$ over $R$, there exists a morphism $\psi_{R(\A,V),R,[M,\phi]}:R(\A,V)\to R$ 
in $\hat{\Ca}$ such that 
\[\hat{\Fun}_V(\psi_{R(\A,V),R,[M,\phi]})[U(\A,V), \phi_{U(\A,V)}]=[M,\phi],\]
and moreover, $\psi_{R(\A,V),R,[M,\phi]}$ is unique if $R$ is the ring of dual numbers $\k[\epsilon]$ with $\epsilon^2=0$.  Then $R(\A,V)$ and $
[U(\A,V),\phi_{U(\A,V)}]$ are called the {\it versal deformation ring} and {\it versal deformation} of $V$, respectively. If the morphism $
\psi_{R(\A,V),R,[M,\phi]}$ is unique for all $R\in\Ob(\hat{\Ca})$ and deformations $[M,\phi]$ of $V$ over $R$, then $R(\A,V)$ and $[U(\A,V),\phi_{U(\A,V)}]$ are 
called the {\it universal deformation ring} and the {\it universal deformation} of $V$, respectively.  In other words, the universal deformation 
ring $R(\A,V)$ represents the deformation functor $\hat{\Fun}_V$ in the sense that $\hat{\Fun}_V$ is naturally isomorphic to the $\Hom$ 
functor $\Hom_{\hat{\Ca}}(R(\A,V),-)$. By \cite[Prop. 2.1]{blehervelez} every finitely generated $\A$-module $V$ has a versal deformation ring $R(\A,V)$. Moreover, if $\End_\A(V)$ is isomorphic to $\k$, then $R(\A,V)$ is universal. Additionally, \cite[Prop. 2.5]{blehervelez2} proves that Morita equivalences preserve isomorphism classes of versal deformation rings. 
\begin{remark}\label{selfuni}
Let $\A$ be a self-injective $\k$-algebra and let $V$ be a finitely generated non-zero left $\A$-module.
\begin{enumerate}
\item It follows from \cite[Thm. 2.6 (ii)]{blehervelez}  that if the stable endomorphism ring of $V$ is isomorphic to $\k$, then the versal deformation ring $R(\A,V)$ is universal.
%\item It follows from \cite[Lemma 3.2.2]{blehervelez2} that the versal deformation rings  $R(\A,V)$ and $R(\A,V\oplus P)$ are isomorphic in $\hat{\Ca}$ for all finitely generated 
%projective left $\A$-modules $P$. 
\item If $\A$ is further a Frobenius $\k$-algebra and $V$ is non-projective, then it follows from \cite[Prop. 2.4]{bleher15} that the versal deformation rings $R(\A,V)$ and $R(\A,\Omega V)$ are isomorphic in $\hat{\Ca}$. Moreover $R(\A, \Omega V)$ is universal if and only if $R(\A,V)$ is universal. Note that this result improves \cite[Thm. 2.6 (iv)]{blehervelez}. 
\end{enumerate}
\end{remark}

\subsection{Some properties of finitely generated Gorenstein-projective modules}

We denote by $\A$-mod the category of finitely generated left $\A$-modules, and by $\A$-\underline{mod} its stable category. We denote by $\A$-Gproj (resp. by $\A$-\underline{Gproj}) the full subcategory of $\A$-mod (resp. of $\A$-\underline{mod}) consisting of finitely generated Gorenstein-projective left $\A$-modules. 

We need the following result that summarizes some properties of finitely generated Gorenstein-projective left $\A$-modules.

\begin{lemma}\label{lemma:3.1}
Let $V$ be an object in $\A\textup{-Gproj}$. 
\begin{enumerate}
\item For all $i\geq 0$, the $i$-th syzygy $\Omega^iV$ of $V$ is also an object in $\A\textup{-Gproj}$.
\item If $P$ is a left $\A$-module with finite projective dimension, then $\Ext_\A^i(V,P)=0$ for all $i>0$.
\item $V$ has finite projective dimension if and only if $V$ is a projective module.
\item Assume that $\A$ is Gorenstein with injective dimension as a left $\A$-module equal to $d\geq 0$. Then $
\A\textup{-Gproj}=\Omega^d(\A\textup{-mod})$. 
\item $\A$-\textup{Gproj} is a Frobenius category in the sense of \cite[Chap. I, \S 2.1]{happel}.
\item $\Omega$ induces an autoequivalence $\Omega:\A\textup{-\underline{Gproj}}\to \A\textup{-\underline{Gproj}}$.

\end{enumerate}
\end{lemma}
\begin{proof} 
Statement (i) follows from \cite[Prop. 2.18]{holmH}, (ii) follows from \cite[Prop. 2.3]{holmH}, (iii) follows from \cite[Prop. 2.27]{holmH}, (iv) follows from \cite[Prop. 3.1(b)]{auslander3}. On the other hand, it is straightforward to prove that $\A$-Gproj has enough projective and injective objects and is closed under extensions, i.e., if $0\to X\to Y\to Z\to 0$ is a short exact sequence of left $\A$-modules with $X$ and $Z$ Gorenstein-projective, then $Y$ is also Gorenstein-projective. In particular, $\A$-Gproj is an exact category in the sense of \cite{quillen}. Since every Gorenstein-projective $\A$-module is in particular (maximal) Cohen-Macaulay (in the sense of \cite{buchweitz}), it follows that the injective and projective objects in $\A$-Gproj coincide, which proves (v). Part (vi) follows from \cite[Chap. I, \S2.2]{happel}.   
\end{proof}
We also need the following well-known result concerning syzygies of $\A$-modules (for a proof see e.g. \cite[Lemma 5.2]{skart}).
\begin{lemma}\label{lemma:5.3}
Let $V$ and $W$ be finitely generated left $\A$-modules, and let $i\geq 1$ be an integer such that $\Ext_\A^i(V,\A)=0$. Then there exists an isomorphism of $\k$-vector spaces 
\begin{equation*}
\Ext_\A^i(V,W)\cong \SHom_\A(\Omega^iV,W).
\end{equation*}
\end{lemma}

\section{(Uni)versal Deformation Rings of Finitely Generated Gorenstein-Projective Modules}\label{section3}
The aim of this section is to prove parts (i) and (ii) of Theorem \ref{thm01}.

We denote by $\Ca$ the full subcategory of $\hat{\Ca}$ consisting of Artinian rings.  Following  \cite[Def. 1.2]{sch}, a {\it small extension} in $\Ca$ is a surjective morphism $\pi:R\to R_0$ in $\Ca$ such that the kernel of $\pi$ is a principal ideal $tR$ annihilated by the maximal ideal $\mathfrak{m}_R$ of $R$. For all surjections  $\pi:R\to R_0$  in $\Ca$, and for all finitely generated projective $R_0\A$-module $Q_0$, we denote by $\mathrm{Proj}_R(Q_0)$ a projective $R\A$-module cover of $Q_0$. It follows that $R_0\otimes_{R,\pi}\mathrm{Proj}_R(Q_0)\cong Q_0$ as $R_0\A$-modules.   

\begin{lemma}\label{claim1}
Let $\pi:R\to R_0$ be a surjection in $\Ca$. 
\begin{enumerate}
\item Let $M$, $Q$ (resp. $M_0$, $Q_0$) be finitely generated $R\A$-modules (resp. $R_0\A$-modules), which are both free over $R$ (resp. $R_0$) and $Q$ (resp. $Q_0$) is projective over $R\A$ (resp. $R_0\A$). Suppose  that $\k\otimes_RM$ is an object in $\A\textup{-Gproj}$, and that there are $R_0\A$-module isomorphisms $g:R_0\otimes_{R,\pi}M\to M_0$ and $h:R_0\otimes_{R,\pi}Q\to Q_0$. If $v_0\in \Hom_{R_0\A}(M_0,Q_0)$, then there exists $v\in \Hom_{R\A}(M,Q)$ with $v_0=h\circ (R_0\otimes_{R,\pi}v)\circ g^{-1}$.
\item Let $M$ (resp. $M_0$) be as in (i). Suppose that  $\sigma_0\in \End_\A(M_0)$ factors through a projective $R_0\A$-module. Then there exists $\sigma\in \End_{R\A}(M)$ such that $\sigma$ factors through a projective $R\A$-module and $\sigma_0=g\circ (R_0\otimes_{R,\pi}\sigma)\circ g^{-1}$.
\item Let $R$ be an Artinian ring in $\Ca$. Suppose $P$ is a finitely generated projective $\A$-module and there exists a commutative diagram of finitely generated $R\A$-modules 
\begin{equation}\label{diag1}
\begindc{\commdiag}[330]
\obj(0,1)[p0]{$0$}
\obj(2,1)[p1]{$\mathrm{Proj}_R(P)$}
\obj(4,1)[p2]{$T$}
\obj(6,1)[p3]{$C$}
\obj(8,1)[p4]{$0$}
\obj(0,-1)[q0]{$0$}
\obj(2,-1)[q1]{$P$}
\obj(4,-1)[q2]{$\k\otimes_RT$}
\obj(6,-1)[q3]{$\k\otimes_RC$}
\obj(8,-1)[q4]{$0$}
\mor{p0}{p1}{}
\mor{p1}{p2}{$\alpha$}
\mor{p2}{p3}{$\beta$}
\mor{p3}{p4}{}
\mor{q0}{q1}{}
\mor{q1}{q2}{$\bar{\alpha}$}
\mor{q2}{q3}{$\bar{\beta}$}
\mor{q3}{q4}{}
\mor{p1}{q1}{}
\mor{p2}{q2}{}
\mor{p3}{q3}{}
\enddc
\end{equation}
in which $T$ and $C$ are free over $R$ and the bottom row arises by tensoring the top row with $\k$ over $R$ and identifying $P$ with $\k\otimes_R\mathrm{Proj}_R(P)$. Assume also that $\k\otimes_RT, \k \otimes_RC$ are objects in $\A\textup{-Gproj}$. Then the top row of (\ref{diag1}) splits as a sequence of $R\A$-modules.
\end{enumerate}
\end{lemma}
\begin{proof}
The proof of Lemma \ref{claim1} is very similar to the proof of Claims 1, 2 and 6 in the proof of \cite[Thm. 2.6]{blehervelez}. As noted in \cite[Remark 3.2.1]{blehervelez2}, one needs the assumption that $M$, $Q$ (resp. $M_0$, $Q_0$) are free over $R$ (resp. $R_0$) in the proof of these claims. The reason for this additional assumption is two-fold. First, this implies that tensoring a projective $R\A$-module resolution of $M$
\begin{equation*}
\cdots \to P_2\xrightarrow{\delta_2}P_1\xrightarrow{\delta_1}P_0\xrightarrow{\delta_0}M\to 0.
\end{equation*} 
with $\k$ over $R$ provides a projective $\A$-module resolution of $\k\otimes_R M$. Second, suppose $\pi:R\to R_0$ is a small extension in $\Ca$ with $\ker \pi=tR$. Then $\Hom_{R\A}(P_i,tQ)$ is isomorphic to $\Hom_\A(\k\otimes_RP_i,tQ)$ for all $i\geq 0$, and these isomorphisms are natural with respect to the $R\A$-module homomorphisms $\delta_i:P_i\to P_{i-1}$ for all $i\geq 1$. 

These observations imply that $\Ext^1_{R\A}(M,tQ)\cong \Ext_\A^1(\k\otimes_RM,tQ)$. By using that $tQ\cong \k\otimes_RQ$ is a projective $\A$-module together with Lemma \ref{lemma:3.1} (ii), we obtain $\Ext_\A^1(\k\otimes_RM,tQ)=0$. The remainder of the proof of Lemma \ref{claim1} (i)-(ii) is the same as the proof of Claims 1 and 2 in the proof of \cite[Thm. 2.6]{blehervelez}.

In order to prove Lemma \ref{claim1} (iii), we use Lemma \ref{lemma:3.1} (ii) again to obtain that $\Ext_\A^1(\k\otimes_RC,P)=0$. This implies that $\bar{\alpha}^\ast:\Hom_\A(\k\otimes_RT,P)\to \Hom_\A(P,P)$ is surjective. Hence there exists a $\A$-module homomorphism $\bar{w}:\k\otimes_RT\to P$ with $\bar{w}\circ \bar{\alpha}=\mathrm{id}_P$. The remainder of the proof of Lemma \ref{claim1} (iii) is the same as the proof of Claim 6 in the proof of \cite[Thm. 2.6]{blehervelez}. 
\end{proof}

The next result is proved in the same way as \cite[Lemma 3.2.2]{blehervelez2} by using Lemma \ref{claim1} instead of \cite[Remark 3.2.1]{blehervelez2}. Note that by \cite[Remark 2.1]{bleher15}, if $P$ is a finitely generated non-zero $\A$-module such that $\Ext_\A^1(P,P)=0$ then the versal deformation ring $R(\A,P)$ is universal and isomorphic to $\k$.

\begin{lemma}\label{lemma3}
Let $V$ be a finitely generated non-zero Gorenstein-projective left $\A$-module and assume that $P$  is a finitely generated left $\A$-module. Then the versal deformation ring $R(\A,P\oplus V)$ is isomorphic to the versal deformation ring $R(\A,V)$.
\end{lemma}

The next result is proved the same way as parts (i) and (ii) of \cite[Thm. 2.6]{blehervelez} by replacing Claims 1 and 2 in the proof of \cite[Thm. 2.6]{blehervelez} by parts (i) and (ii) of Lemma \ref{claim1}.
\begin{theorem}\label{thm1}
Let $V$ be a finitely generated Gorenstein-projective left $\A$-module whose stable endomorphism ring $\SEnd_\A(V)$ is isomorphic to $\k$.
\begin{enumerate}
\item The deformation functor $\hat{\Fun}_V$ is naturally isomorphic to the weak deformation functor $\hat{\Fun}_V^w$ as in Remark \ref{rem1}.
\item The module $V$ has a universal deformation ring $R(\A,V)$. 
\end{enumerate}
\end{theorem}

\section{(Uni)versal Deformation Rings of Finitely Generated Gorenstein-Projective Modules and Singular Equivalences of Morita Type}\label{section4}
The aim of this section is to prove Theorem \ref{thm01} (iii).
Let $\A$ and $\Gamma$ be two finite dimensional $\k$-algebras.
\begin{remark}\label{rem:5.3}
The concept of singular equivalence of Morita type, as given in Definition \ref{defi:3.2},  was further generalized by Z. Wang in \cite{wang}, where the concept of {\it singular 
equivalence of Morita type with level} is introduced. It was proved by \O. Skarts{\ae}terhagen in \cite[Prop. 2.6]{skart} that if ${_\Gamma}X_\A$ and ${_\A}Y_\Gamma$ are bimodules that induce a singular equivalence of Morita type, then they induce a singular equivalence of Morita type with level. Therefore, it follows from \cite[Lemma. 3.6]{skart} that if ${_\Gamma}X_\A$ and ${_\A}Y_\Gamma$ are bimodules which induce a singular equivalence of Morita type between two finite-dimensional Gorenstein $\k$-algebras $\A$ and $\Gamma$ as in Definition \ref{defi:3.2}, then the functors 
\begin{align*}%\label{cmequiv0}
X\otimes_\A-:\A\textup{-mod} \to \Gamma \textup{-mod}&& \text{ and } && Y\otimes_\Gamma-:\Gamma\textup{-mod} \to \A\textup{-mod} 
\end{align*}
send finitely generated Gorenstein-projective left modules to finitely generated Gorenstein-projective left modules. By \cite[Prop. 2.3]{zhouzimm} and \cite[Prop. 3.7]{skart} it follows that 
\begin{align*}%\label{cmequiv}
X\otimes_\A-:\A\textup{-\underline{Gproj}}\to \Gamma\textup{-\underline{Gproj}}&& \text{ and } && Y\otimes_\Gamma-:
\Gamma\textup{-\underline{Gproj}}\to \A\textup{-\underline{Gproj}} 
\end{align*}
are equivalences of triangulated categories that are quasi-inverses of each other.
\end{remark}

\begin{remark}\label{rem:4.2}
Assume that $\A$ and $\Gamma$ are both Gorenstein $\k$-algebras, and that ${_\Gamma}X_\A$ and ${_\A}Y_\Gamma$ are bimodules that induce a singular equivalence of Morita type between $\A$ and $\Gamma$ as in Definition \ref{defi:3.2}. Moreover, let $P$ be a $\A$-$\A$-bimodule with finite projective dimension such that $Y\otimes_\Gamma X\cong \A\oplus P$ as $\A$-$\A$-bimodules, as in Definition \ref{defi:3.2} (iv). 
\begin{enumerate}
\item If $V$ is a finitely generated Gorenstein-projective left $\A$-module, then we have by Remark \ref{rem:5.3} that $Y\otimes_\Gamma(X\otimes_\A V)\cong V\oplus (P\otimes_\A V)$ and $V$ are isomorphic in the stable category $\A$-\underline{Gproj}. This implies that $P\otimes_\A V$ is a finitely generated projective left $\A$-module. 
\item Let $R\in \Ob(\Ca)$ be Artinian. Then $X_R=R\otimes_\k X$ is projective as a left $R\Gamma$-module and as a right $R\A$-module, and $Y_R=R\otimes_\k Y$ is projective 
as a left $R\A$-module and as a right $R\Gamma$-module. Note that $X_R\otimes_{R\A}Y_R\cong R\otimes_\k(X\otimes_\A Y)$ as $R\Gamma$-$R\Gamma$-bimodules and $Y_R\otimes_{R\Gamma} X_R\cong R\otimes_\k (Y\otimes_\A X)$ as $R\A$-$R\A$-bimodules. Therefore, it follows from Definition \ref{defi:3.2} that 
\begin{align*}
X_R\otimes_{R\A} Y_R&\cong R\Gamma\oplus Q_R&&\text{ as $R\Gamma$-$R\Gamma$-bimodules, and}\\
Y_R\otimes_{R\Gamma} X_R&\cong R\A \oplus P_R&&\text{ as $R\A$-$R\A$-bimodules,}
\end{align*}
where $P_R=R\otimes_\k P$ (resp. $Q_R=R\otimes_\k Q$) is an $R\A$-$R\A$-bimodule (resp. $R\Gamma$-$R\Gamma$-bimodule) with 
finite projective dimension.
\end{enumerate}
\end{remark}
The next result shows that singular equivalences of Morita type preserve versal deformation rings. This is proved the same way as \cite[Prop. 3.2.6]{blehervelez2} by using Remark \ref{rem:4.2} and by replacing \cite[Lemma 3.2.2]{blehervelez2} by Lemma \ref{lemma3}. Note that the last statement is a direct consequence of Theorem \ref{thm1}. 

\begin{theorem}\label{prop2.7}
Assume that $\A$ and $\Gamma$ are Gorenstein $\k$-algebras. Suppose that ${_\Gamma}X_\A$ and ${_\A} Y_\Gamma$ are bimodules that induce a singular equivalence of Morita type between $\A$ and $\Gamma$ as in Definition \ref{defi:3.2}. Let $V$ be a finitely generated Gorenstein-projective left $\A$-module, and define $V'=X\otimes_\A V$. Then the versal deformation rings $R(\A,V)$ and $R(\Gamma, V')$ are 
isomorphic in $\hat{\Ca}$. Moreover, if the stable endomorphism ring of $V$ is isomorphic to $\k$, then both $R(\A,V)$ and $R(\Gamma, V')$ are universal deformation rings that are isomorphic.  
\end{theorem}

\section{Examples}\label{section5}

In this section we provide two applications of Theorem \ref{thm01}. Namely, we first consider monomial $\k$-algebras with no overlap (see Theorem \ref{thmmon}). Then we illustrate part (iii) of Theorem \ref{thm01} by discussing an example of two Gorenstein $\k$-algebras (both non-self-injective of infinite global dimension), which are stably equivalent of Morita type, and thus singularly equivalent of Morita type as in Definition \ref{defi:3.2} (see Example \ref{exam1}).

Recall that a quiver $Q$ is a directed graph with a set of vertices $Q_0$, a set of arrows $Q_1$ and two functions $\mathbf{s},\mathbf{t}:Q_1\to Q_0$, where for all $\alpha\in Q_1$, $\mathbf{s}\alpha$ (resp. $\mathbf{t}\alpha$) denotes the vertex where $\alpha$ starts (resp. ends). A path in $Q$ of length $n\geq 1$ is an ordered sequence of arrows $p=\alpha_n\cdots\alpha_1$ with $\mathbf{t}\alpha_j=\mathbf{s}\alpha_{j+1}$ for $1\leq j <n$. Additionally, for each $i\in Q_0$, we have a trivial path $e_i$ of length zero with $\mathbf{s}e_i=i=\mathbf{t}e_i$. For a non-trivial path $p=\alpha_n\cdots \alpha_1$ we define $\mathbf{s}p=\mathbf{s}\alpha_1$ and $\mathbf{t}p=\mathbf{t}\alpha_n$.  A non-trivial path $p$ in $Q$ is said to be an oriented cycle provided that $\mathbf{s}p=\mathbf{t}p$. The path algebra $\k Q$ of a quiver $Q$ is the $\k$-vector space whose basis consists of all the paths in $Q$, and for two paths $p$ and $q$, their multiplication is given by the concatenation $pq$ provided that $\mathbf{s}p=\mathbf{t}q$, or zero otherwise. Let $J$ be the two-sided ideal of $\k Q$ generated by all the arrows in $Q$. We say that an ideal $I$ of $\k Q$ is admissible if there exists $d\geq 2$ such that $J^d\subseteq I \subseteq J^2$. In this situation, the quotient $\k Q/ I$ is a finite dimensional $\k$-algebra. If $p$ is a path in $Q$, we denote also by $p$ its equivalence class in $\k Q/I$. In particular, a path $p$ in $\k Q/I$ is a {\it zero-path} if and only if $p$ belongs to $I$. Recall that an admissible ideal $I$ of $\k Q$ is said to be monomial if it is generated by paths of length at least two. In this situation we say that the quotient algebra $\k Q/I$ is a {\it monomial algebra}. Recall that a monomial algebra $\k Q/I$ is said to be {\it quadratic monomial} provided that the ideal $I$ is generated by paths of length two. In particular, gentle algebras (as introduced in \cite{assem}) are quadratic monomial. 

Let $\A=\k Q/I$ be a monomial algebra. Following \cite{chen-shen-zhou}, we say that a pair $(p,q)$ of non-zero paths in $\A$ is a {\it perfect pair} provided that the following conditions are satisfied:
\begin{itemize}
\item[(P1)] both $p$ and $q$ are non-trivial with $\mathbf{s}p=\mathbf{t}q$ and $pq$ is a zero-path in $\A$;
\item[(P2)] if $pq'$ is a zero-path in $\A$ for a non-zero path $q'$ with $\mathbf{t}q'=\mathbf{s}p$, then $q'=qq''$ for some path $q''$ in $\A$;
\item[(P3)] if $p'q$ is a zero-path in $\A$ for a non-zero path $p'$ with $\mathbf{t}q=\mathbf{s}p'$, then $p'=p''p$ for some path $p''$ in $\A$.
\end{itemize}
A non-zero path $p$ in $\A$ is {\it perfect}, provided that there exists a sequence $p=p_1, p_2,\ldots, p_n, p_{n+1}=p$ of non-zero paths in $\A$ such that for all $1\leq i\leq n$, the pair $(p_i, p_{i+1})$ is a perfect pair. 
It follows from \cite[Thm. 4.1]{chen-shen-zhou} that a finitely generated indecomposable non-projective left $\A$-module $V$ is Gorenstein-projective if and only if $V= \A p$, where $p$ is a perfect path in $\A$. 
%In this situation, if $\Omega V=\Omega \A p$ denotes the first syzygy of $V=\A p$, then it follows from \cite[Lemma 3.1]{chen-shen-zhou} that there exists a non-zero path $q$ in $\A$ such that $\Omega V =\A q$, and thus we obtain a short exact sequence of left $\A$-modules  
%\begin{equation}\label{eqn1}
%0\to \A q \xrightarrow{\iota} \A e_{\mathbf{t} p}\xrightarrow{\pi} \A p \to 0. 
%\end{equation}
%Note that $\A q$ is also a non-projective Gorenstein-projective left $\A$-module. By \cite[Prop. 4.3]{chen-shen-zhou}, we can assume that $q$ is also a perfect path in $\A$. 
Following \cite[\S 5]{chen-shen-zhou}, an {\it overlap} in $\A$ is given by two perfect paths $p$ and $q$ in $\A$ that satisfy one of the following conditions:
\begin{itemize}
\item[(O1)] $p=q$, and $p=p'x$ and $q=xq'$ for some non-trivial paths $x$, $p'$ and $q'$ with the path $p'xq'$ non-zero.
\item[(O2)] $p\not=q$, and $p=p'x$ and $q=xq'$ for some non-trivial path $x$ with the path $p'xq'$ non-zero.  
\end{itemize}

\begin{remark}\label{remmon}
Assume that there is no overlap in $\A$ and let $V$ be a finitely generated  indecomposable Gorenstein-projective left $\A$-module. If  $V$ is non-projective, then by the arguments in the proof of \cite[Prop. 5.9]{chen-shen-zhou} together with Lemma \ref{lemma:5.3} we obtain that $\SEnd_\A(V)=\k$ and 
\begin{equation}\label{ext2}
\Ext_\A^1(V,V)= \SHom_\A(\Omega V, V)=
\begin{cases}
\k, &\text{ if $V=\Omega V$,}\\
0, &\text{ otherwise.}
\end{cases}
\end{equation}
\end{remark}

\begin{theorem}\label{thmmon}
Let $\A=\k Q/I$ be a monomial algebra in which there is no overlap, and let $V$ be a finitely generated indecomposable Gorenstein-projective left $\A$-module. Then the versal deformation ring $R(\A,V)$ of $V$ is universal and isomorphic either to $\k$ or to $\k[\![t]\!]/(t^2)$. 
\end{theorem}

\begin{proof}
Assume that $\A$ is a monomial algebra in which there is no overlap, and let $V$ be a finitely generated indecomposable Gorenstein-projective left $\A$-module. If $\Ext_\A^1(V,V)=0$, then it follows by \cite[Remark 2.1]{bleher15} that the versal deformation ring $R(\A,V)$ is universal and isomorphic to $\k$. Assume next that $\Ext_\A^1(V,V)\not=0$. By Remark \ref{remmon}, this means that $\Omega V=V$. By (\ref{ext2}), we obtain $\Ext_\A^1(V,V)\cong \k$, which implies that $R(\A,V)$ is isomorphic to a quotient algebra of $\k[\![t]\!]$.  Consider then the corresponding short exact sequence of left $\A$-modules
\begin{equation*}
0\to V \xrightarrow{\iota_V} P(V)\xrightarrow{\pi_V} V \to 0, 
\end{equation*}
where $\pi_V:P(V)\to V$ is a projective cover of $V$. 
It follows that $P(V)$ defines a non-trivial lift of $V$ over the ring of dual numbers $\k[\![t]\!]/(t^2)$, where the action of $t$ is given by $\iota_V \circ \pi_V$. This implies that there exists a unique surjective $\k$-algebra morphism $\psi: R(\A, V)\to \k[\![t]\!]/(t^2)$ in $\hat{\Ca}$ corresponding to the deformation defined by $P(V)$. We need to show that $\psi$ is an isomorphism. Suppose otherwise. Then there exists a surjective $\k$-algebra homomorphism $\psi_0:R(\A,V)\to \k[\![t]\!]/(t^3)$ in $\hat{\Ca}$ such that $\pi'\circ \psi_0=\psi$, where $\pi':\k[\![t]\!]/(t^3)\to \k[\![t]\!]/(t^2)$ is the natural projection.  Let $M_0$ be a $\k[\![t]\!]/(t^3)\A$-module which defines a lift of $V$ over $\k[\![t]\!]/(t^3)$ corresponding to $\psi_0$.  Let $(U(\A,V),\phi_{U(\A,V)})$ be a lift of $V$ over $R(\A,V)$ that defines the universal deformation of $V$. Then $M_0\cong \k[\![t]\!]/(t^3)\otimes_{R(\A,V), \psi_0}U(\A,V)$. Note that $M_0/tM_0\cong V$ as $\A$-modules. On the other hand, we also have that 
\begin{align*}
P(V)&\cong \k[\![t]\!]/(t^2)\otimes_{R(\A,V), \psi}U(\A,V)\\
&\cong \k[\![t]\!]/(t^2)\otimes_{\k[\![t]\!]/(t^3), \pi'}(\k[\![t]\!]/(t^3)\otimes_{R(\A,V), \psi_0}U(\A,V))\\
&\cong \k[\![t]\!]/(t^2)\otimes_{\k[\![t]\!]/(t^3),\pi'}M_0.
\end{align*}
Note that since $\ker \pi'=(t^2)/(t^3)$, we have $\k[\![t]\!]/(t^2)\otimes_{\k[\![t]\!]/(t^3),\pi'}M_0\cong M_0/t^2M_0$. Thus $P(V)\cong M_0/t^2M_0$ as $\k[\![t]\!]/(t^2)\A$-modules. Consider the surjective $\k[\![t]\!]/(t^3)\A$-module homomorphism $g:M_0\to t^2M_0$ defined by $g(x)=t^2x$ for all $x\in M_0$. Since $M_0$ is free over $\k[\![t]\!]/(t^3)$, it follows that $\ker g =tM_0$ and thus $M_0/tM_0\cong t^2M_0$, which implies that $V\cong t^2M_0$. Hence we get a short exact sequence of $\k[\![t]\!]/(t^3)\A$-modules 
\begin{equation}\label{ext4}
0\to V\to M_0\to P(V)\to 0.
\end{equation} 
Since $P(V)$ is a projective left $\A$-module, it follows that (\ref{ext4}) splits as a short exact sequence of $\A$-modules. Hence $M_0=V\oplus P(V)$ as $\A$-modules. Writing elements of $M_0$ as $(u,v)$  where $u\in V$ and $v\in P(V)$, the $t$-action on $M_0$ is given as $t(u,v)=(\mu(v), tv)$ for some surjective $\A$-module homomorphism $\mu: P(V)\to V$. Using $t^2v=0$ for all $v\in P(V)$, we obtain that $t^2(u,v)=(\mu(tv),0)$ for all $u\in V$ and $v\in P(V)$. Since $t^2M_0\cong V$, this means that the restriction of $\mu$ to $tP(V)$ has to define an isomorphism $\tilde{\mu}:tP(V)\to V$. Therefore, $\tilde{\mu}^{-1}$ provides a $\A$-module homomorphism splitting of $\mu$, which shows that $V$ is isomorphic to a direct summand of $P(V)$. But then $V$ is itself projective, which contradicts our assumption that $\Ext_\A^1(V,V)\not=0$. Thus $\psi$ is a $\k$-algebra isomorphism and $R(\A,V)\cong \k[\![ t]\!]/(t^2)$. 
This finishes the proof of Theorem \ref{thmmon}.
\end{proof}

\begin{remark}
If $\A$ is a quadratic monomial algebra, then there is no overlap in $\A$ since all perfect paths are  arrows. Therefore Theorem \ref{thmmon} applies to quadratic monomial algebras, and consequently to gentle algebras. 
\end{remark}

\begin{example}\label{exam1}
Let $\A=\k Q/\langle \rho\rangle$ and $\Gamma=\k Q'/\langle \rho'\rangle$ be the $\k$-algebras whose quivers with relations are given as follows:
\begin{align*}
Q&: \xymatrix@1@=20pt{\underset{1}{\bullet}\ar@/^/[r]^{\alpha}&\underset{2}{\bullet}\ar@/^/[l]^{\beta}},&& \rho=\{\beta\alpha\beta\alpha\},\\
Q'&: \xymatrix@1@=20pt{\underset{1'}{\bullet}\ar@/^/[r]^{x}&\underset{2'}{\bullet}\ar@/^/[l]^{y}\ar@(ur,dr)^{z}},&& \rho'=\{yx,zx,yz,z^2-xy\}.
\end{align*}
By \cite[\S 5]{liu-xi}, $\A$ and $\Gamma$ are stably equivalent of Morita type and their self-injective dimensions on both sides are equal to $2$. It follows that $\A$ and $\Gamma$ are Gorenstein $\k$-algebras that are singularly equivalent of Morita type in the sense of Definition \ref{defi:3.2}. One easily verifies that both of them are not self-injective and have infinite global dimension. On the other hand, by \cite[Example 5.10]{chen-shen-zhou}, $\A$ is a monomial algebra in which there is no overlap and $\beta\alpha$ is the unique perfect path in $\A$. It follows that $V=\A\beta\alpha$ is the unique (up to isomorphism) finitely generated indecomposable non-projective Gorenstein-projective left $\A$-module. 
Because  $\Omega V=V$, it follows by Theorem \ref{thmmon} that $R(\A,V)$ is universal and isomorphic to $\k[\![t]\!]/(t^2)$. Since $\A$ and $\Gamma$ are singularly equivalent of Morita type and since $\SEnd_\A(V)=\k$, it follows that there is a unique (up to isomorphism) finitely generated indecomposable non-projective Gorenstein-projective left $\Gamma$-module $W$, which (up to addition of projective modules) corresponds to $V$ under the singular equivalence of Morita type. By Theorem \ref{thm01} (iii), we should have that $R(\Gamma,W)\cong \k[\![t]\!]/(t^2)$.  We can verify this as follows. Since $\Gamma$ is special biserial, we have a description of all indecomposable $\Gamma$-modules (up to isomorphism), using strings and bands (see \cite{buri}). Denote by $P_{1'}$ and $P_{2'}$ the indecomposable projective left $\Gamma$-modules corresponding to the vertices of $Q'$, and consider $M[z^{-1}x]$ the string left $\Gamma$-module induced by the string representative $z^{-1}x$ for $\Gamma$. It is straightforward to show that $\Omega M[z^{-1}x]=M[z^{-1}x]$, which clearly implies that $\Omega^d M[z^{-1}x]= M[z^{-1}x]$ for all $d\geq 2$. Since $\Gamma$ is Gorenstein with injective dimension as a left $\Gamma$-module equal to $2$, it follows, using Lemma \ref{lemma:3.1} (iv), that $M[z^{-1}x]$ is a finitely generated indecomposable non-projective Gorenstein-projective left $\Gamma$-module. This means that $W=M[z^{-1}x]$. By using the description of the morphisms between string modules as in \cite{krause} together with Lemma \ref{lemma:5.3}, we obtain that $\SEnd_{\Gamma}(M[z^{-1}x])=\k$ and $\Ext_{\Gamma}^1(M[z^{-1}x],M[z^{-1}x])=\k$. This implies that $R(\Gamma, M[z^{-1}x])$ is universal and isomorphic to a quotient algebra of $\k[\![t]\!]$. By using the short exact sequence of left $\Gamma$-modules that defines $\Omega M[z^{-1}x]=M[z^{-1}x]$, namely,
\begin{equation*}
0\to M[z^{-1}x]\to P_{1'}\oplus P_{2'}\to M[z^{-1}x]\to 0, 
\end{equation*}
and by arguing as in the proof of Theorem \ref{thmmon}, we obtain that $R(\Gamma, M[z^{-1}x])\cong \k[\![t]\!]/(t^2)$. This illustrates Theorem \ref{thm01} (iii).
\end{example}

\begin{remark}\label{remfinal}
Let $\A$ be a finite dimensional $\k$-algebra and let $V$ be a left $\A$-module with finite dimension over $\k$ and such that $\SEnd_\A(V)=\k$. Assume that $\A$ is a basic self-injective $\k$-algebra. It follows from e.g. \cite[Prop. IV.3.9]{skow3} that $\A$ is also Frobenius. Then by Remark \ref{selfuni}, we have that the versal deformation rings $R(\A,V)$ and $R(\A, \Omega V)$ are both universal and isomorphic in $\hat{\Ca}$. In view of Theorem \ref{thm01} (ii), when $\A$ is an arbitrary finite dimensional $\k$-algebra and $V$ is a finitely generated Gorenstein-projective left $\A$-module, we have that both versal deformation rings $R(\A, V)$ and $R(\A, \Omega V)$ are universal. However, we do not know at this point if in general they are also isomorphic in $\hat{\Ca}$ or not.  Looking at the algebras $\A$ in Theorem \ref{thmmon} and Example \ref{exam1}, a natural question to ask seems to be whether this is true more generally, or at least when $\A$ is a Gorenstein $\k$-algebra.
\end{remark}
\section{Acknowledgments}  
This article was developed when the third author was a Visiting Associate Professor of Mathematics at the Instituto of Matem\'aticas at the Universidad de Antioquia in Medell{\'\i}n, Colombia during the summer of 2016. The third author would like to express his gratitude to the other authors, faculty members, staff and students at the Instituto of Matem\'aticas as well as to the other people related to this work at the 
Universidad de Antioquia for their hospitality and support during his visit. The authors want to express their gratitude to F. M. Bleher, who made some suggestions after reading earlier versions of this article, to X. W. Chen for sending the preprint \cite{chensun} to them, and to the anonymous referee, who provided many suggestions and corrections that helped the readability and quality of this work, and who also recommended to look at the article \cite{chen-shen-zhou}, which was used to prove Theorem \ref{thmmon}.

%\nocite{*}
\bibliographystyle{amsplain}
\bibliography{BekkertGiraldoVelezGorensteinRev7}

\end{document}